\documentclass[11pt,dvips,twoside]{article}

\usepackage{pslatex}
\usepackage{fancyhdr}
\usepackage{graphicx}
\usepackage{geometry}

\RequirePackage{amsfonts,amssymb,amsmath,amscd,amsthm}
\RequirePackage{txfonts}
\RequirePackage{graphicx}
\RequirePackage{xcolor}
\RequirePackage{geometry}
\RequirePackage{enumerate}

\def\figurename{Figure} % Replace the colon that normally appears after the Figure number by a period.
\makeatletter
\renewcommand{\fnum@figure}[1]{\figurename~\thefigure.}
\makeatother

\def\tablename{Table} % Replace the colon that normally appears after the Figure number by a period.
\makeatletter
\renewcommand{\fnum@table}[1]{\tablename~\thetable.}
\makeatother

\usepackage{amsmath}
\usepackage{amssymb}
\usepackage{amsfonts}
\usepackage{amsthm,amscd}

\newtheorem{theorem}{Theorem}[section]
\newtheorem{lemma}[theorem]{Lemma}
\newtheorem{corollary}[theorem]{Corollary}
\newtheorem{proposition}[theorem]{Proposition}

\theoremstyle{example}
\newtheorem{example}[theorem]{Example}

\theoremstyle{definition}
\newtheorem{definition}[theorem]{Definition}

\theoremstyle{remark}
\newtheorem{remark}[theorem]{Remark}

\numberwithin{equation}{section}

\begin{document}

\title{\bfseries\scshape{BiHom-Poisson color algebras}}

\author{\bfseries\scshape Ibrahima BAKAYOKO\thanks{e-mail address: ibrahimabakayoko27@gmail.com}\\
D\'epartement de Math\'ematiques,
Universit\'e de N'Z\'er\'ekor\'e,\\
BP 50 N'Z\'er\'ekor\'e, Guin\'ee.
}
 
\date{}
\maketitle 

% % % % % % \thispagestyle{empty} \setcounter{page}{1}
% % % % % % % ------- [First Page Running Head] - place it immediately after title! ------
% % % % % % \thispagestyle{fancy} \fancyhead{}
% % % % % % \fancyhead[L]{{\LARGE A}frican {\LARGE D}iaspora {\LARGE J}ournal of {\LARGE M}athematics\\
% % % % % % Volume X, Number X, pp. {\thepage--\pageref{lastpage-01} (2013)}} % put \label{lastpage-xx} on the last page!
% % % % % % \fancyhead[R]{ISSN 1539-854X  \\ {\tt{www.math-res-pub.org/adjm}}}
% % % % % % \fancyfoot{}
% % % % % % \renewcommand{\headrulewidth}{0pt}
%------------------------------------------------------------------------------

\noindent\hrulefill

\noindent {\bf Abstract.} 
The goal of this paper is to introduce  BiHom-Poisson color algebras and give various constructions by using some specific maps such as 
morphisms. We introduce averaging operator and element of centroid for  BiHom-Poisson color algebras and point out that BiHom-Poisson
 color algebras are closed under averaging operators and elements of centroid. We also show that any regular BiHom-associative color algebra leads 
to  BiHom-Poisson color algebra via the commutative bracket. Then we prove that any  BiHom-Poisson color algebra together 
with Rota-Baxter operator or multiplier give rises to another  BiHom-Poisson color algebra. Next, we show that tensor 
product of any BiHom-associative color algebra and any  BiHom-Poisson color algebra is also a  BiHom-Poisson color algebra.

\noindent \hrulefill

\vspace{.3in}


\vspace{.08in} \noindent \textbf{Keywords}: 
 BiHom-Poisson color algebras, bijective even linear map, element of centroid, averaging operator,  Rota-Baxter operator, 
multiplier and tensor product.
\vspace{.3in}
% \noindent {\small Revised: April 16, 2014, June 28, 2014 $\parallel$  Accepted: July 4, 2014}
\vspace{.2in}
\section{Introduction}
BiHom-algebraic structures were introduced in the first time in 2015 by G. Graziani, A. Makhlouf, C. Menini and F. Panaite 
 in \cite{GACF} from a categorical approach  as an extension of the class of Hom-algebras.
Since then, other interesting BiHom-type algebraic structures of many Hom-algebraic structures has been intensively studied as 
BiHom-Lie colour algebras structures \cite{ABA}, Representations of BiHom-Lie algebras \cite{YH}, 
BiHom-Lie superalgebra structures \cite{SS}, 
$\{\sigma, \tau\}$-Rota-Baxter operators, infinitesimal
Hom-bialgebras and the associative (Bi)Hom-Yang-Baxter equation \cite{MFP}, The construction and deformation of BiHom-Novikov algebras \cite{SZW},
On n-ary Generalization of BiHom-Lie algebras
and BiHom-Associative Algebras \cite{KMS}, Rota-Baxter operators on BiHom-associative
algebras and related structures \cite{LAC}.

The purpose of this paper is to introduce  BiHom-Poisson color algebras and give various constructions by using special even linear maps. 
In section 2, we give basic definition and properties of  BiHom-Poisson color algebras. 
In particular, we show that any regular BiHom-associative color algebra leads 
to  BiHom-Poisson color algebra via the commutative bracket.
In section 3, we show that BiHom-Poisson
 color algebras are closed under either tensorization with scaler field or commutative BiHom-associative color algebras, averaging operators and elements of centroid. 
Moreover, we prove that any  BiHom-Poisson color algebra endowed with Rota-Baxter operator or multiplier give rises to another 
 BiHom-Poisson color algebra.
% ------------ [Running Heads - for odd and even pages] - please insert it only on page 2!
\pagestyle{fancy} \fancyhead{} \fancyhead[EC]{ } 
\fancyhead[EL,OR]{\thepage} \fancyhead[OC]{Ibrahima bakayoko} \fancyfoot{}
\renewcommand\headrulewidth{0.5pt}
%------------------------------------------------------------------------------

 Throughout this paper, all graded vector spaces are assumed to be over a field $\mathbb{K}$ of characteristic different from 2.
\section{Definitions and basic properties}
In this section, we recall basic definitions and some elementary properties.
\subsection{Multipliers}
\begin{definition}
 Let $G$ be an abelian group. A map $\varepsilon :G\times G\rightarrow {\bf \mathbb{K}^*}$ is called a skew-symmetric bicharacter on $G$ if the following
identities hold, 
\begin{enumerate} 
 \item [(i)] $\varepsilon(g, g')\varepsilon(g', g)=1$,
\item [(ii)] $\varepsilon(g, g'+g'')=\varepsilon(g, g')\varepsilon(g, g'')$,
\item [(iii)]$\varepsilon(g+g', g'')=\varepsilon(g, g'')\varepsilon(g', g'')$,
\end{enumerate}
$g, g', g''\in G$.
\end{definition}
\begin{remark}
Observe that $\varepsilon(g, e)=\varepsilon(e, g)=1, \varepsilon(g,g)=\pm 1 \;\mbox{for all}\; g\in G, \;\mbox{where}\; e \;\mbox{is the identity of}\; G.$
\end{remark}

\begin{example}
 \begin{enumerate}
\item [1)] 
$ G=\mathbb{Z}_2^n=\{(\alpha_1, \dots, \alpha_n)| \alpha_i\in\mathbb{Z}_2 \}, \quad
\varepsilon((\alpha_1, \dots, \alpha_n), (\beta_1, \dots, \beta_n)):= (-1)^{\alpha_1\beta_1+\dots+\alpha_n\beta_n},$
\item [2)] $G=\mathbb{Z}\times\mathbb{Z} ,\quad \varepsilon((i_1, i_2), (j_1, j_2))=(-1)^{(i_1+i_2)(j_1+j_2)}$,
\item [3)] $G=\{-1, +1\} , \quad\varepsilon(i, j)=(-1)^{(i-1)(j-1)/{4}}$.
\end{enumerate}
\end{example}

\begin{example}
 Let $\sigma : G\times G\rightarrow\mathbb{K}^*$ be any mapping such that 
\begin{eqnarray}
 \sigma(g, g'+g'')\sigma(g', g'')=\sigma(g, g')\sigma(g+g', g''), \forall g, g', g''\in G.
\end{eqnarray}
Then, $\delta(g, g')=\sigma(g, g')\sigma(g', g)^{-1}$ is a bicharacter on $G$. In this case, $\sigma$ is called a {\it multiplier} on $G$, and 
$\delta$ the bicharacter associated with $\sigma$.

For instance, let us define the mapping $\sigma : G\times G\rightarrow\mathbb{R}$ by
\begin{eqnarray}
 \sigma((i_1, i_2), (j_1, j_2))=(-1)^{i_1j_2}, \forall i_k, j_k\in\mathbb{Z}_2, k=1, 2.\nonumber
\end{eqnarray}
It is easy to verify that $\sigma$ is a multiplier on $G$ and 
\begin{eqnarray}
 \delta((i_1, i_2), (j_1, j_2))=(-1)^{i_1j_2-i_2j_1}, \forall i_k, j_k\in \mathbb{Z}_2, i=1, 2. \nonumber
\end{eqnarray}
is a bicharacter on $G$.
\end{example}

\begin{definition}
A color BiHom-algebra is a quadruple $(A, \mu, \varepsilon, \alpha)$ in which 
\begin{enumerate}
 \item [a)] $A$ is a $G$-graded vector space i.e. $A=\bigoplus_{g\in G}A_g$,
\item [b)] $\mu : A \times A \rightarrow A$ is a even bilinear map i.e. $\mu(A_g, A_{g'})\subseteq A_{g+g'}$, for all $g, g'\in G$,
\item [c)] $\alpha, \beta : A\rightarrow A$ are even linear maps i.e. $\alpha(A_g)\subseteq A_{g}, \beta(A_g)\subseteq A_{g}$,
\item [d)] $\varepsilon : G\times G\rightarrow{\bf K}^*$ is a bicharacter.
\end{enumerate}
\end{definition}

If x and y are two homogeneous elements of degree $g$ and $g'$ respectively and $\varepsilon$ is a skew-symmetric bicharacter, 
then we shorten the notation by writing $\varepsilon(x, y)$ instead of $\varepsilon(g, g')$.
 \subsection{BiHom-Poisson color algebras}
In this subsection, we introduce  BiHom-Poisson color algebra and give some basic properties.

\begin{definition}(\cite{BIA})
  A Hom-Poisson color  algebra consists of a $G$-graded vector space $A$, a multiplication $\mu : A\times A\rightarrow A$,
 an even bilinear bracket $\{\cdot, \cdot\} : A\times A\rightarrow A$ and an even linear map $\alpha : A\rightarrow A$ such that :
\begin{enumerate}
 \item[1)] $(A, \mu, \varepsilon, \alpha)$ is a Hom-associative color algebra,
\item [2)]$(A, \{\cdot, \cdot\}, \varepsilon, \alpha)$ is a Hom-Lie color  algebra,
\item[3)]the Hom-Leibniz color identity
\begin{eqnarray}
 \{\alpha(x), \mu(y, z)\}=\mu(\{x, y\}, \alpha(z))+\varepsilon(x,y)\mu(\alpha(y), \{x, z\}),\label{cca}\nonumber
\end{eqnarray}
is satisfied for any $x, y, z\in \mathcal{H}(A)$.
\end{enumerate}
A Hom-Poisson color algebra $(A, \mu, \{\cdot, \cdot\}, \varepsilon, \alpha)$ in which $\mu$ is $\varepsilon$-commutative
is said to be a commutative Hom-Poisson color algebra.
 \end{definition}

The following definition is motivated by the one above.
\begin{definition}
  A BiHom-Poisson color  algebra consists of a $G$-graded vector space $P$, a multiplication $\mu : P\times P\rightarrow P$,
 a bilinear bracket $\{\cdot, \cdot\} : P\times P\rightarrow P$ and even linear maps $\alpha, \beta : P\rightarrow P$ such that :
\begin{enumerate}
 \item[1)] $(P, \mu, \varepsilon, \alpha, \beta)$ is a BiHom-associative color algebra of degree zero i.e. :
\begin{eqnarray}
\alpha\circ\beta&=&\beta\circ\alpha\\
 as_\mu(x, y, z)&=&\mu(\alpha(x), \mu(y, z))-\mu(\mu(x, y), \alpha(z))=0, 
\end{eqnarray}
for any $x, y, z\in \mathcal{H}(P)$.
\item [2)]$(P, \{\cdot, \cdot\}, \varepsilon, \alpha, \beta)$ is a BiHom-Lie color algebra i.e. :
\begin{eqnarray}
\alpha\circ\beta=\beta\circ\alpha
\end{eqnarray}
\begin{eqnarray}\alpha([x, y])=[\alpha(x), \alpha(y)], \quad \beta([x, y])=[\beta(x), \beta(y)]
\end{eqnarray}
\begin{eqnarray}
{[\beta(x), \alpha(y)]}=-\varepsilon(x, y+)[\beta(y), \alpha(x)],\; \mbox{\it (BiHom-$\varepsilon$-skew-simmetry color identity)}
\end{eqnarray}
\begin{eqnarray}
{\bf J}(x, y, z)= \oint_{x, y, z}\varepsilon(z, x+)[\beta^2(x), [\beta(y), \alpha(z)]]=0,\; \mbox{\it (BiHom-Jacobi color identity)}
\end{eqnarray}
for any $x, y, z\in \mathcal{H}(P)$, where $\oint_{x, y, z}$ means cyclic summation over ${x, y, z}$.
\item[3)]the {\it BiHom-Leibniz color identity}
\begin{eqnarray}
 \{\alpha(x), \mu(y, z)\}=\mu(\{x, y\}, \alpha(z))+\varepsilon(x,y)\mu(\alpha(y), \{x, z\}),
\end{eqnarray}
is satisfied for any $x, y, z\in \mathcal{H}(P)$.
\end{enumerate}
 \end{definition}

\begin{definition}
 \begin{enumerate}
  \item [a)] A  BiHom-Poisson color algebra $(P, \mu, \{\cdot, \cdot\}, \varepsilon, \alpha, \beta)$ in which $\mu$ 
is $\varepsilon$-commutative i.e.
$$\mu(x, y)=\varepsilon(x, y)\mu(y, x)$$
is said to be a commutative  BiHom-Poisson color algebra.
\item [b)] A  BiHom-Poisson color algebra $(P, \mu, \{\cdot, \cdot\}, \varepsilon, \alpha, \beta)$ in which $\alpha$ and $\beta$ are
 automorphisms is called regular  BiHom-Poisson color algebra.
 \end{enumerate}
\end{definition}

\begin{example}
 Any BiHom-associative color algebra \cite{ADR} can be seen as a BiHom-Poisson color algebra with the trivial BiHom-Lie color structure, and 
any BiHom-Lie color algebra \cite{ABA} can be seen as a  BiHom-Poisson color algebra with the trivial color BiHom-associative product. 
\end{example}

\begin{example}
Hom-Poisson color algebras \cite{BIA} or Poisson color algebras \cite{AD} are examples of BiHom-Poisson color algebras
by setting $\beta = id$ or $\alpha=id$ and $\beta=id$. If, in addition, $\varepsilon(x, y)=1$ or $\varepsilon(x, y)=(-1)^{|x||y|}$, then the
 BiHom-Poisson color
algebra is nothing but classical Poisson algebras \cite{} or Poisson superalgebras.
 BiHom-Poisson algebras \cite{XL} and BiHom-Poisson superalgebras are also obtained when $\varepsilon(x, y)=1$ and $\varepsilon(x, y)=(-1)^{|x||y|}$
 respectively. If, moreover $\beta = id$  we get Hom-Poisson algebras \cite{DYa} and Hom-Poisson superalgebras.
\end{example}

\begin{example}(\cite{IT})
 If $(P, [-, -], \cdot, \star, \ast, \varepsilon, \alpha)$ is a  Hom-post-Poisson color algebra. Then $(P, \odot, \{-,-\}, \varepsilon, \alpha)$ 
is a commutative Hom-Poisson color  algebra with
\begin{eqnarray}
x\odot y=x\star y+\varepsilon(x, y)y\star x+x\ast y \quad \mbox{and}\quad \{x, y\}=x\cdot y-\varepsilon(x, y)y\cdot x+[x, y]\nonumber
\end{eqnarray}
for any $x, y\in {P}$. .
\end{example}

The next result allows to construct a  BiHom-Poisson color algebra from a given one and a bijective map.

\begin{theorem}
 Let $(P', \cdot', [-, -]', \varepsilon, \alpha', \beta')$ be a  BiHom-Poisson color algebra and $P$ a graded vector space
 with a $\varepsilon$-skew-symmetric bilinear bracket of degree and even linear maps $\alpha$ and $\beta$. Let $f : L\rightarrow L'$ be an even bijective linear 
map such that $f\circ\alpha=\alpha'\circ f, f\circ\beta=\beta'\circ f$, 
$$f(x\cdot y)=f(x)\cdot' f(y)\quad\mbox{and}\quad f([x, y])=[f(x), f(y)]', \forall x, y\in \mathcal{H}(P).$$
Then $(P, \cdot, [-, -], \varepsilon, \alpha, \beta)$ is a  BiHom-Poisson color algebra.
\end{theorem}
\begin{proof}
BiHom-associativity and BiHom-Jacobi color identities are checked as in the following BiHom-Leibniz color identity.
Then, for any $x, y, z\in \mathcal{H}(P)$,
\begin{eqnarray}
 [\beta\alpha(x), y\cdot z]
&=&f^{-1}[f(\beta\alpha(x)), f(y\cdot z)]'\nonumber\\
&=&f^{-1}[f(\beta\alpha(x)), f\Big(f^{-1}\Big(f(y)\cdot' f(z)\Big)\Big)]'\nonumber\\
&=&f^{-1}[\beta'\alpha'(f(x)), f(y)\cdot' f(z)]'\nonumber\\
&=&f^{-1}\Big([f(\beta(x), f(y)]'\cdot'\beta'(f(z))+\varepsilon(x, y) \beta'(f(y))\cdot' [f(\alpha(x))\cdot' f(z)]'\Big)\nonumber\\
&=&f^{-1}\Big(f\Big(f^{-1}[f(\beta(x)), f(y)]'\Big)\cdot'\beta'(f(z))\Big)\nonumber\\
&&\quad+\varepsilon(x, y) f^{-1}\Big(\beta'(f(y))\cdot' f\Big(f^{-1}[f(\alpha(x))\cdot' f(z)]'\Big)\Big)\nonumber\\
&=&f^{-1}\Big(f([\beta(x), y])\cdot'f(\beta(z))\Big)
+\varepsilon(x, y) f^{-1}\Big(f(\beta(y))\cdot' f([\alpha(x)\cdot z])\Big)\nonumber\\
&=&[\beta(x), y]\cdot\beta(z)+\varepsilon(x, y) \beta(y)\cdot [\alpha(x)\cdot z]\nonumber.
\end{eqnarray}
This gives the conclusion.
\end{proof}

The following theorem asserts that the commutator of any BiHom-associative color algebra gives rise to BiHom-Poisson color algebra.
\begin{theorem}
 Let $(A, \cdot, \varepsilon, \alpha, \beta )$ be a regular BiHom-associative color algebra. Then
$$P(A)=(A, \cdot, [-, -], \varepsilon, \alpha, \beta)$$
is a regular BiHom-Poisson color algebra, where
$$[x, y]=x\cdot y-\varepsilon(x, y)\alpha^{-1}\beta(y)\cdot\alpha\beta^{-1}(x),$$
for any $x, y\in\mathcal{H}(A).$
\end{theorem}
\begin{proof}
It is proved in (Proposition 1.6 \cite{ABA}) that any BiHom-associative color algebra carries a structure of a BiHom-Lie color algebra. Thus, we only
need to check the compatibility condition.
For any homogeneous elements $x, y, z\in A$,
  \begin{eqnarray}
[\alpha\beta(x), y\cdot z]&=& \alpha\beta(x)(yz)-\varepsilon(x, y+z)\alpha^{-1}\beta(yz)\alpha\beta^{-1}\alpha\beta(x)  \nonumber\\
&=& \alpha\beta(x)(yz)-\varepsilon(x, y+z)(\alpha^{-1}\beta(y) \cdot\alpha^{-1}\beta(z))\alpha^2(x).  \nonumber
  \end{eqnarray}
  \begin{eqnarray}
[\beta(x), y]\cdot\beta(z)&=&\Big(\beta(x)\cdot y-\varepsilon(x, y)\alpha^{-1}\beta(y)\cdot\alpha\beta^{-1}\beta(x)\Big)\cdot\beta(z)\nonumber\\
&=&(\beta(x)\cdot y)\cdot\beta(z) -\varepsilon(x, y)(\alpha^{-1}\beta(y)\cdot\alpha(x))\cdot\beta(z)\nonumber.
  \end{eqnarray}
and
  \begin{eqnarray}
\beta(y)\cdot[\alpha(x), z]&=&\beta(y)\cdot \Big(\alpha(x)\cdot z-\varepsilon(x, z)\alpha^{-1}\beta(z)\cdot\alpha\beta^{-1}\alpha(x)\Big)\nonumber\\
&=&\beta(y)\cdot (\alpha(x)\cdot z)-\varepsilon(x, z)\beta(y)\cdot(\alpha^{-1}\beta(z)\cdot\alpha^2\beta^{-1}(x))   \nonumber\\
&=&\alpha(\alpha^{-1}\beta(y))\cdot (\alpha(x)\cdot z)
-\varepsilon(x, z)\alpha(\alpha^{-1}(\beta(y))\cdot(\alpha^{-1}\beta(z)\cdot\alpha^2\beta^{-1}(x))   \nonumber\\
&=&(\alpha^{-1}\beta(y)\cdot \alpha(x))\cdot \beta(z)
-\varepsilon(x, z)(\alpha^{-1}\beta(y)\cdot\alpha^{-1}\beta(z))\cdot\alpha^2(x) \nonumber.
  \end{eqnarray}

Therefore,
  \begin{eqnarray}
[\alpha\beta(x), yz]= [\beta(x), y]\cdot\beta(z)+\varepsilon(x, y)\beta(y)\cdot[\alpha(x), z]\nonumber.  
  \end{eqnarray}
This concludes the proof.
\end{proof}

\section{Main results}
This section is devoted to an exposition of various constructions of  BiHom-Poisson color algebras.
\subsection{BiHom-Poisson color algebras arisng from twisting}
The following theorem asserts that  BiHom-Poisson color algebra turn to another one via morphisms.
\begin{theorem}\label{tw}
 Let $(P, \mu, [-, -], \varepsilon, \alpha, \beta )$ be a  BiHom-Poisson color algebra and $\alpha', \beta' : P\rightarrow P$ two even morphisms of 
BiHom-Poisson color such that any two of the maps $\alpha, \alpha', \beta, \beta'$ commute. Then
$$P_{(\alpha', \beta')}=(P, \ast=\mu(\alpha'\otimes\beta'), \{-, -\}:=[-,-](\alpha'\otimes\beta'), \varepsilon, \alpha\alpha', \beta\beta')$$
is a  BiHom-Poisson color algebra. 
\end{theorem}
\begin{proof}
We only prove the BiHom-Leibniz color identity and leave the BiHom-associativity and BiHom-Lie identities to the reader.
 \begin{eqnarray}
 \{\alpha\alpha'\beta\beta'(x), y\ast z\}&=&[\alpha\alpha'^2\beta\beta'(x), \beta'(\alpha'(y)\beta'(z))]\nonumber\\
&=&[\alpha\alpha'^2\beta\beta', \alpha'\beta'(y)\beta'^2(z)]\nonumber\\
&=&[\alpha\beta(\alpha'^2\beta')(x), \alpha'\beta'(y)\beta'^2(z)]\nonumber\\
&=&[\beta(\alpha'^2\beta'(x)), \alpha'\beta'(y)]\beta\beta'^2(z)
+\varepsilon(x, y)\beta(\alpha'\beta'^2(y))[\alpha(\alpha'^2\beta'(x)), \beta'^2(z)]\nonumber\\
&=&\alpha'[\alpha'(\beta\beta'(x)), \beta'(y)]\beta'(\beta\beta'(z))
+\varepsilon(x, y)\alpha'(\beta\beta'(y))\beta'[\alpha'(\alpha\alpha'(x)), \beta'(z)]\nonumber\\
&=&\{\beta\beta'(x), y\}\ast\beta\beta'(z)+\varepsilon(x, y)\beta\beta'(y)\ast\{\alpha\alpha'(x), z\}\nonumber.
 \end{eqnarray}
This ends the proof.
\end{proof}

\begin{corollary}
 Let $(P, \mu, [-, -], \varepsilon, \alpha, \beta )$ be a  BiHom-Poisson color algebra. Then
$$(P, \mu\circ(\alpha^n\otimes\beta^n), [-, -]\circ(\alpha^n\otimes\beta^n), \alpha^{n+1}, \beta^{n+1})$$
is also a  BiHom-Poisson color algebra.
\end{corollary}
\begin{proof}
 It suffises to take $\alpha'=\alpha^n$ and $\beta'=\beta^n$ in Theorem \ref{tw}.
\end{proof}

By taking $\alpha'=\beta'=\alpha\neq id$ in Theorem \ref{tw}, we get the following corollary.

 \begin{corollary}
 Let $(P, \mu, [-, -], \varepsilon, \alpha)$ be a Hom-Poisson color algebra. Then
$$(P, \mu\circ(\alpha\otimes\beta), [-, -]\circ(\alpha\otimes\beta), \alpha^{2}, \beta^{2})$$
is also a Hom-Poisson color algebra.
\end{corollary}

Any regular Hom-Poisson color algebra give rises to Poisson color algebra as stated in the next corollary.
\begin{corollary}
 Let $(P, \mu, [-, -], \varepsilon, \alpha)$ be a regular Hom-Poisson color algebra. Then
$$(P, \mu\circ(\alpha^{-1}\otimes\beta^{-1}), [-, -]\circ(\alpha^{-1}\otimes\beta^{-1}))$$
is also a Poisson color algebra.
\end{corollary}

 \begin{corollary}
 Let $(P, \mu, [-, -], \varepsilon)$ be a Poisson color algebra. Then
$$(P, \mu\circ(\alpha\otimes\beta), [-, -]\circ(\alpha\otimes\beta), \alpha, \beta)$$
is a  BiHom-Poisson color algebra.
\end{corollary}

\subsection{Extensions of  BiHom-Poisson color algebras}
In this subsection, we give extensions of a given  BiHom-Poisson color algebra by a field, by commutative associative color algebras or by 
another BiHom-Poisson color algebra.

\begin{theorem}
Let $(P, \cdot, [-, -], \varepsilon, \alpha)$ be a  BiHom-Poisson color algebra over a field $\mathbb{K}$ and $\hat{\mathbb{K}}$
 an extension of $\mathbb{K}$.  Then, the graded $\mathbb{K}$-vector space 
 $$\hat{\mathbb{K}}\otimes P=\sum_{g\in G}(\mathbb{K}\otimes P)_g=\sum_{g\in G}\mathbb{K}\otimes P_g$$
is an  BiHom-Poisson color algebra with :\\ 
the associative product
$$ (r\otimes x)\cdot'(s\otimes y):=rs\otimes (x\cdot y),$$
the bracket
$$[r\otimes x,  s\otimes y]'=rs\otimes [x, y],$$
the even linear maps
$$\alpha'(r\otimes x):=r\otimes \alpha(x) \quad\mbox{and}\quad \beta'(r\otimes x):=r\otimes \beta(x),$$
and the bicharacter
$$\varepsilon(r+x, s+y)=\varepsilon(x, y), \forall r, s\in\hat{\mathbb{K}} , \forall x, y\in \mathcal{H}(P).$$
\end{theorem}
\begin{proof}
 It is proved by a straightforward computation.
\end{proof}

We need the following Lemmas for the next theorem.
\begin{lemma}\label{bhca1}
In any commutative BiHom-associative color algebra, we have
\begin{enumerate}
 \item [1)] $\beta^2(a)(\beta(b)\alpha(c))=\beta^2(a)(\beta(b)\beta(c)),$
 \item [2)] $\beta^2(b)(\beta(c)\alpha(a))=\varepsilon(b+c, a)\beta^2(a)(\beta(b)\beta(c)),$ 
\item [3)]$\beta^2(c)(\beta(a)\alpha(b))=\varepsilon(c, a+b)\beta^2(a)(\beta(b)\beta(c)),$  
\end{enumerate}
\end{lemma}
\begin{proof}
\begin{enumerate}
 \item [1)] For any $a, b, c\in \mathcal{H}(A)$, we have :
\begin{eqnarray}
 \beta^2(a)(\beta(b)\alpha(c))&=&\varepsilon(a, b+c)(\beta(b)\alpha(c))\beta^2(a)=\varepsilon(a, b+c)\alpha\beta(b)(\alpha(c)\beta(a))\nonumber\\
&=&\varepsilon(a, b+c)\varepsilon(c, a)\alpha\beta(b)(\beta(a)\alpha(c))=\varepsilon(a, b)(\beta(b)\beta(a))\beta\alpha(c)\nonumber\\
&=&\varepsilon(a, b)\varepsilon(a+b, c)\alpha\beta(c)(\beta(b)\beta(a))=\varepsilon(a, b)\varepsilon(a+b, c)(\beta(c)\beta(b))\beta^2(a)\nonumber\\
&=&\beta^2(a)(\beta(b)\beta(c))\nonumber.
\end{eqnarray}
\item [2)] For any $a, b, c\in \mathcal{H}(A)$, we have :
\begin{eqnarray}
 \beta^2(b)(\beta(c)\alpha(a))
&=&\varepsilon(b, a+c)(\beta(c)\alpha(a))\beta^2(b)
=\varepsilon(b, a+c)\alpha\beta(c)(\alpha(a)\beta(b))\nonumber\\
&=&\varepsilon(b, a+c)\varepsilon(a, b)\alpha\beta(c)(\beta(b)\alpha(a))
=\varepsilon(b, c)(\beta(c)\beta(b))\beta\alpha(a)\nonumber\\
&=&\varepsilon(b, c)\varepsilon(b+c, a)\alpha\beta(a)(\beta(c)\beta(b))
=\varepsilon(b, c)\varepsilon(b+c, a)(\beta(a)\beta(c))\beta^2(b)\nonumber\\
&=&\varepsilon(b, a+c)(\beta(c)\beta(a))\beta^2(b)
=\varepsilon(b, a+c)\alpha\beta(c)(\beta(a)\beta(b))\nonumber\\
&=&\varepsilon(b, c)\alpha\beta(c)(\beta(b)\beta(a))
=\varepsilon(b, c)(\beta(c)\beta(b))\beta^2(a)\nonumber\\
&=&\varepsilon(b, c)\varepsilon(b+c, a)\varepsilon(c, b)\beta^2(a)(\beta(b)\beta(c))\nonumber\\
&=&\varepsilon(b+c, a)\beta^2(a)(\beta(b)\beta(c))\nonumber.
\end{eqnarray}
 \item [3)] For any $a, b, c\in \mathcal{H}(A)$, we have :
\begin{eqnarray}
 \beta^2(c)(\beta(a)\alpha(b))&=&\varepsilon(c, a+b)(\beta(a)\alpha(b))\beta^2(c)=\varepsilon(c, a+b)\alpha\beta(a)(\alpha(b)\beta(c))\nonumber\\
&=&\varepsilon(c, a+b)\varepsilon(b, c)\alpha\beta(a)(\beta(c)\alpha(b))=\varepsilon(c, a)(\beta(a)\beta(c))\beta\alpha(b)\nonumber\\
&=&\varepsilon(a+c, b)\varepsilon(c, a)\alpha\beta(b)(\beta(a)\beta(c))=\varepsilon(a+c, b)\alpha\beta(b)(\beta(c)\beta(a)) \nonumber\\
&=&\varepsilon(a+c, b)(\beta(b)\beta(c))\beta^2(a)=\varepsilon(a+c, b)\varepsilon(b+c, a)\beta^2(a)(\beta(b)\beta(c))\nonumber\\
&=&\varepsilon(c, a+b)\beta^2(a)(\beta(b)\beta(c))\nonumber.
\end{eqnarray}
\end{enumerate}
This ends the proof.
\end{proof}

\begin{lemma}\label{bhca2}
If $(A, \cdot, \varepsilon, \alpha, \beta)$ is commutative BiHom-associative color algebra, we have
\begin{eqnarray}
 \alpha\beta(a)(bc)= \varepsilon(a, b)\beta(b)(\alpha(a)c).
\end{eqnarray}
\end{lemma}
\begin{proof}
For any $a, b, c\in \mathcal{H}(A)$, we have :
\begin{eqnarray}
 \alpha\beta(a)(bc)&=&\varepsilon(a, b+c)(bc)\beta\alpha(a)=\varepsilon(a, b+c)\alpha(b)(c\alpha(a))
=\varepsilon(a, b+c)\varepsilon(c, a)\alpha(b)(\alpha(a)c)\nonumber\\
&=&\varepsilon(a, b)(b\alpha(a))\beta(c)=(\alpha(a)b)\beta(c)=\alpha^2(a)(bc)=\varepsilon(b, c)\alpha^2(a)(cb)\nonumber\\
&=&\varepsilon(b, c)(\alpha(a)c)\beta(b)=\varepsilon(b, c)\varepsilon(a+c, b)\beta(b)(\alpha(a)c)\nonumber\\
&=&\varepsilon(a, b)\beta(b)(\alpha(a)c)\nonumber.
\end{eqnarray}
This finishes the proof.
\end{proof}

The below result gives an extension of a  BiHom-Poisson color algebra by a commutative BiHom-associative color algebra.
\begin{theorem}\label{tens}
 Let $(A, \cdot, \varepsilon, \alpha, \beta)$  be a commutative BiHom-associative
color algebra and $(P, \ast, \{-, -\}, \varepsilon, \alpha_L)$ be a  BiHom-Poisson color algebra. 
Then the tensor product $A\otimes P$ endowed with the even (bi)linear maps
\begin{eqnarray} 
\alpha &:&=\alpha_A\otimes \alpha_P : A\otimes P\rightarrow A\otimes P,\quad a\otimes x\mapsto\alpha_A(a)\otimes\alpha_P(x), \nonumber\\
\beta &:&=\beta_A\otimes \beta_P : A\otimes P\rightarrow A\otimes P,\quad a\otimes x\mapsto\beta_A(a)\otimes\beta_P(x) ,\nonumber\\
\ast &:& (A\otimes P)\times (A\otimes P)\rightarrow A\otimes P, 
\quad (a\otimes x, b\otimes y)\mapsto\varepsilon(x, b)a\cdot b\otimes x\ast y,\nonumber\\
\{-, -\} &:& (A\otimes P)\times (A\otimes P)\rightarrow A\otimes P, \quad (a\otimes x, b\otimes y)\mapsto\varepsilon(x, b)(a\cdot b)\otimes [x, y],\nonumber
\end{eqnarray}
is a  BiHom-Poisson color algebra.
\end{theorem}
\begin{proof}
To simplify the typography, we make no distinction among $\alpha, \alpha_A, \alpha_P$ or $\beta, \beta_A, \beta_P$.
It is easy to prove the color BiHom-associativity.
For the BiHom-Jacobi color identity, we have, for any $a, b, c\in \mathcal{H}(A)$ and $x, y, z\in \mathcal{H}(P)$,
\begin{eqnarray}
 [\beta^2(a\otimes x), [\beta(b\otimes y), \alpha(c\otimes z)]]
&=&\varepsilon(y, c)[\beta^2(a)\otimes\beta^2(x), \beta(b)\alpha(c)\otimes[\beta(y), \alpha(z)]]\nonumber\\
&=&\varepsilon(y, c)\varepsilon(x, b+c)\beta^2(a)(\beta(b)\alpha(c))\otimes[\beta^2(x), [\beta(y), \alpha(z)]]\nonumber.
\end{eqnarray}
By Lemma \ref{bhca1},
\begin{eqnarray}
\oint\varepsilon(c+z, a+x) [\beta^2(a\otimes x), [\beta(b\otimes y), \alpha(c\otimes z)]]=0\nonumber.
\end{eqnarray}
 Now,
 for any $a, b, c\in \mathcal{H}(A)$ and $x, y, z\in \mathcal{H}(P)$,
\begin{eqnarray}
 \{\beta\alpha(a\otimes x), (b\otimes y)\ast(c\otimes z)\}
&=&\varepsilon(y, c)\{\alpha\beta(a)\otimes\beta\alpha(x), bc\otimes yz\}\nonumber\\
&=&\varepsilon(y, c)\varepsilon(x, b+c)\alpha\beta(a)(bc)\otimes[\beta\alpha(x), yz]\nonumber\\
&=&\varepsilon(y, c)\varepsilon(x, b+c)\alpha\beta(a)(bc)\otimes([\beta(x), y]\beta(z)+\varepsilon(x, y)\beta(y)[\alpha(x), z])\nonumber\\
&=&\varepsilon(y, c)\varepsilon(x, b) \varepsilon(x, c)\alpha\beta(a)(bc)\otimes[\beta(x), y]\beta(z)+\nonumber\\
&&+\varepsilon(y, c)\varepsilon(x, b+c)\varepsilon(x, y)\alpha\beta(a)(bc)\otimes\beta(y)[\alpha(x), z]\nonumber.
\end{eqnarray}

By color BiHom-associativity and Lemma \ref{bhca2},
\begin{eqnarray}
 \{\beta\alpha(a\otimes x), (b\otimes y)\ast(c\otimes z)\}
&=&\varepsilon(x, b)(\beta(a)b\otimes[\beta(x), y])(\beta(c)\otimes\beta(z))\nonumber\\
&&+\varepsilon(y, c)\varepsilon(x, b+c)\varepsilon(x, y)\varepsilon(a, b)\beta(b)(\alpha(a)c)\otimes\beta(y)[\alpha(x), z]\nonumber\\
&=&\varepsilon(x, b)(\beta(a)b\otimes[\beta(x), y])(\beta(c)\otimes\beta(z))\nonumber\\
&&+\varepsilon(x, c)\varepsilon(x, b)\varepsilon(x, y)\varepsilon(a, b)\varepsilon(a, y)(\beta(b)\otimes\beta(y))(\alpha(a)c\otimes[\alpha(x), z])\nonumber\\
&=&\{\beta(a)\otimes\beta(x), b\otimes y\}\beta(c)\otimes\beta(z)\nonumber\\
&&+\varepsilon(a+x, b+y)(\beta(b)\otimes\beta(y)) \{\alpha(a)\otimes\alpha(x), c\otimes z\}\nonumber\\
&=&\{\beta(a\otimes x), b\otimes y\}\beta(c\otimes z)
+\varepsilon(a+x, b+y)\beta(b\otimes y) \{\alpha(a\otimes x), c\otimes z\}.\nonumber
\end{eqnarray}
This completes the proof.
\end{proof}
From Theorem \ref{tens}, we obtain Theorem 2.17 \cite{BIA}.
\begin{corollary}
  Let $(A, \cdot, \varepsilon)$  be a commutative Hom-associative
color algebra and $(P, \ast, \{-, -\}, \varepsilon, \alpha, \beta)$ be a  BiHom-Poisson color algebra. 
Then the tensor product $A\otimes P$ is a  BiHom-Poisson color algebra.
\end{corollary}
\begin{corollary}
  Let $(A, \cdot, \varepsilon)$  be a commutative associative
color algebra and $(P, \ast, \{-, -\}, \varepsilon, \alpha, \beta)$ be a  BiHom-Poisson color algebra. 
Then the tensor product $A\otimes P$ is a  BiHom-Poisson color algebra.
\end{corollary}

As the tensor product of two BiHom-Poisson color algebra fails to be a BiHom-Poisson color algebra, we have however the following theorem :
\begin{proposition}\label{tensh}
Let $(P_1, \ast_1,  [-, -]_1, \varepsilon, \alpha_1)$ and $(P_2, \ast_2,  [-, -]_2, \varepsilon, \alpha_2)$ be two
Hom-Poisson color algebras such that the maps 
\begin{eqnarray}
 \alpha(a\otimes x)& :=&\alpha_1(a)\otimes\alpha_2(x),\\
(a\otimes x)\ast(b\otimes y)& :=&\varepsilon(x, b)a\ast_1b\otimes x\ast_2y,\\
{[a\otimes x, b\otimes y]}& :=&\varepsilon(x, b)\Big(a\ast_1b\otimes[x, y]_2+[a, b]\otimes x\ast_2 y\Big),
\end{eqnarray}
satisfy the Hom-Jacobi identity on $P_1\otimes P_2$. Then, $(P_1\otimes P_2, \ast, [-, -], \alpha)$ is a  Hom-Poisson color algebra.
\end{proposition}
\begin{proof}
 The proof is long and straightforward but uses the same technics as in Theorem \ref{tens}.
\end{proof}
\begin{corollary}(\cite{BH} Proposition 2.5.2)
The tensor product of two Poisson color algebras is again a Poisson color algebra.
\end{corollary}
\begin{corollary}(\cite{DYa} Theorem 2.9)
The tensor product of two Hom-Poisson algebras is also a Hom-Poisson algebra.
\end{corollary}
% \newpage
\subsection{BiHom-Poisson color algebras induced by Rota-Baxter operators}

\begin{definition}
Let $(P, \cdot, [-, -], \varepsilon, \alpha, \beta)$ be a  BiHom-Poisson color algebra.  An even linear map $R : P\rightarrow P$ is called
a Rota-Baxter operator of weight $\lambda\in\mathbb{K}$ on $P$ if
\begin{eqnarray}
\alpha\circ R&=&R\circ\alpha,\quad \beta\circ R=R\circ\beta,\\
 R(x)\cdot R(y)&=&R\Big(R(x)\cdot y+x\cdot R(y)+\lambda x\cdot y\Big),\\ 
 {[R(x), R(y)]}&=&R\Big([R(x), y]+[x, R(y)]+\lambda[x, y]\Big),
\end{eqnarray}
for all $x, y\in\mathcal{H}(A)$.
 \end{definition}

The below result connects  BiHom-Poisson color algebras to Rota-Baxter operator.
\begin{theorem}
Let $(P, \cdot,  [-, -], \varepsilon, \alpha, \beta)$ be a  BiHom-Poisson color algebra and $R : P\rightarrow P$ be  Rota-Baxter operator of weight
 $\lambda\in{\bf K}$ on $P$. Then $P$ is also a  BiHom-Poisson color algebra with
\begin{eqnarray}
x\ast y &=& R(x)\cdot y + x\cdot R(y) +\lambda x\cdot y,\\
\{x, y\} &=& [R(x), y] + [x, R(y)] +\lambda [x, y],
\end{eqnarray}
for all $x, y\in\mathcal{H}(P)$.\\
Moreover, $R$ is a morphism of  BiHom-Poisson color algebra of $(P, \ast,  \{x, y\}, \varepsilon, \alpha, \beta)$\\
 into $(P, \cdot,  [-, -], \varepsilon, \alpha, \beta)$.
\end{theorem}
\begin{proof}
 For any $x, y, z\in \mathcal{H}(P)$,
\begin{eqnarray}
 \{\beta\alpha(x), y\ast z\}
&=&\{\beta\alpha(x), R(y)z+yR(z)+\lambda yz\}\nonumber\\
&=&[R\beta(x), R(y)z+yR(z)+\lambda yz]+[\beta(x), R(R(y)z+yR(z)+\lambda yz)]\nonumber\\
&&+\lambda[\beta(x), R(y)z+yR(z)+\lambda yz]\nonumber\\
&=&[R\beta(x), R(y)z]+[R\beta(x), yR(z)]+\lambda[R\beta(x), yz]+[\beta(x), R(y)R(z)]\nonumber\\
&&+\lambda[\beta(x), R(y)z]+\lambda[\beta(x), yR(z)]+\lambda^2[\beta(x), yz]\nonumber.
\end{eqnarray}
By BiHom-Leibniz color identity,
\begin{eqnarray}
 \{\beta\alpha(x), y\ast z\}
&=&[R\beta(x), R(y)]\beta(z)+\varepsilon(x, y)\beta(R(y))[R\alpha(x), z]  +[R\beta(x), y]\beta(R(z))\nonumber\\
&&+\varepsilon(x, y)\beta(y)\cdot[R\alpha(x), R(z)]+\lambda[R\beta(x), y]\beta(z)+\lambda\varepsilon(x, y)\beta(y)\cdot[R\alpha(x), z]\nonumber\\
&&+[\beta(x), R(y)]\beta(R(z))+\varepsilon(x, y)\beta(R(y))[\alpha(x), R(z)]+\lambda[\beta(x), R(y)]\beta(z)\nonumber\\
&&+\varepsilon(x, y)\lambda\beta R(y)[\alpha(x), z]+\lambda[\beta(x), y]\beta(R(z))+\lambda\varepsilon(x, y)\beta(y)[\alpha(x), R(z)]\nonumber\\
&&+\lambda^2[\beta(x), y]\beta(z)+
\lambda^2\varepsilon(x, y)\beta(y)[\alpha(x), z]\nonumber.
\end{eqnarray}
By reorganizing the terms, we have
\begin{eqnarray}
 \{\beta\alpha(x), y\ast z\}
&=&[R\beta(x), R(y)]\beta\alpha(z)+[R\beta(x), y]\beta(R(z))+[\beta(x), R(y)]\beta(R(z))+\lambda[\beta(x), y]\beta(R(z))\nonumber\\
&&+\lambda[R\beta(x), y]\beta(z)+\lambda[\beta(x), R(y)]\beta(z)+\lambda^2[\beta(x), y]\beta(z)\nonumber\\
&&+\varepsilon(x, y)\beta(R(y))[R\alpha(x), z]+\varepsilon(x, y)\beta(R(y))[\alpha(x), R(z)]\nonumber\\
&&+\varepsilon(x, y)\lambda\beta R(y)[\alpha(x), z]+\varepsilon(x, y)\beta(y)[R\alpha(x), R(z)]\nonumber\\
&&+\lambda\varepsilon(x, y)\beta(y)[R\alpha(x), z]
+\lambda\varepsilon(x, y)\beta(y)[\alpha(x), R(z)]\nonumber\\
&&+\lambda^2\varepsilon(x, y)\beta(y)[\alpha(x), z]\nonumber
\end{eqnarray}
\begin{eqnarray}
&=&R\Big([R\beta(x), y]+[\beta(x), R(y)]+\lambda[\beta(x), y]\Big)\beta(z)\nonumber\\
&&+\Big([R\beta(x), y]+[\beta(x), R(y)]+\lambda[\beta(x), y]\Big)R\beta(z)\nonumber\\
&&+\lambda\Big([R\beta(x), y]+[\beta(x), R(y)]+\lambda[\beta(x), y]\Big)\beta(z)\nonumber\\
&&+\varepsilon(x, y)R\beta(y)\Big([R\alpha(x), z]+[\alpha(x), R(z)]+\lambda[\alpha(x), z]\Big)
+\varepsilon(x, y)\beta(y)[R\alpha(x), R(z)]\nonumber\\
&&+\lambda\varepsilon(x, y)\beta(y)\Big([R\alpha(x), z]+[\alpha(x), R(z)]+\lambda[\alpha(x), z]\Big).\nonumber\\
&=&\Big([R\beta(x), y]+[\beta(x), R(y)]+\lambda[\beta(x), y]\Big)\ast\beta(z)\nonumber\\
&&+\varepsilon(x, y)\beta(y)\ast\Big([R\alpha(x), z]+[\alpha(x), R(z)]+\lambda[\alpha(x), z]\Big).\nonumber\\
&=&\{\beta(x), y\}\ast\beta(z)+\varepsilon(x, y)\beta(y)\ast\{\alpha(x), z\}.\nonumber
\end{eqnarray}
The Hom-associativity and Hom-Jacobi color identity are proved in a similar way.
\end{proof}

Whenever $\lambda=0$, $\varepsilon\equiv 1$, $G=\{e\}$, we recover
\begin{corollary}\cite{XL}
 Let $(P, \cdot,  [-, -], \alpha, \beta)$ be a BiHom-Poisson algebra and $R : P\rightarrow P$ be
  Rota-Baxter operator of weight $0$ on $P$. Then $P_R=(P, \ast, \{-, -\}, \alpha, \beta)$ is also a BiHom-Poisson algebra.
Moreover, $R$ is a morphism of BiHom-Poisson algebra of $(P, \ast,  \{x, y\}, \alpha, \beta)$
 into $(P, \cdot,  [-, -],  \alpha, \beta)$.
\end{corollary}

With $\beta=\alpha$, we get
\begin{corollary}
If $(P, \cdot,  [-, -], \varepsilon, \alpha)$ be a Hom-Poisson color algebra and $R : P\rightarrow P$ be  Rota-Baxter operator of weight
 $\lambda\in{\bf K}$. Then $P_R=(P, \ast, \{-, -\}, \alpha)$ is also a Hom-Poisson color algebra.
Moreover, $R$ is a morphism of Hom-Poisson color algebra of $(P, \ast,  \{x, y\}, \varepsilon, \alpha, \beta)$\\
 into $(P, \cdot,  [-, -], \varepsilon, \alpha, \beta)$.
\end{corollary}

\subsection{BiHom-Poisson color algebras induced by centroid and averaging operators}

Now, let us introduce averaging operator for  BiHom-Poisson color algebras.
\begin{definition}
Let $(P, \cdot, [-, -], \varepsilon, \alpha, \beta)$ be a  BiHom-Poisson color algebra. For any positive integers $k, l$,
  an even linear map $\theta : P\rightarrow P$ is called an $(\alpha^k, \beta^l)$-averaging operator over $P$ if
\begin{eqnarray}
\alpha\circ\theta&=&\theta\circ\alpha,\quad \beta\circ\theta=\theta\circ\beta,\\
 \theta(\theta(x)\cdot \alpha^k\beta^l(y))&=&\theta(x)\cdot\theta(y))=\theta(\alpha^k\beta^l(x)\cdot \theta(y)),\\ 
 \theta([\theta(x), \alpha^k\beta^l(y)])&=&[\theta(x), \theta(y)]=\theta([\alpha^k\beta^l(x), \theta(y)]),
\end{eqnarray}
for all $x, y\in\mathcal{H}(A)$.
 \end{definition}
For example, in the Hom-setting, any even $\alpha$-differential operator $d : A\rightarrow A$ (i.e. an $\alpha$-derivation $d$ such that $d^2=0$) over a Hom-associative
 color algebra  is an  $\alpha$-averaging operator.

\begin{theorem}
 Let $(P, \cdot, [-, -], \varepsilon, \alpha)$ be  BiHom-Poisson color algebra and $\theta : P\rightarrow P$ be an 
 $\alpha^0$-averaging operator. Then the two new products
\begin{eqnarray}
 x\ast y=\theta(x)\cdot \theta(y)\quad\mbox{and}\quad \{x, y\}=[\theta(x), \theta(y)], \forall x, y\in \mathcal{H}(P)
\end{eqnarray}
makes $(P, \ast, \{-, -\}, \varepsilon, \alpha)$ a  BiHom-Poisson color algebra.\\
 \end{theorem}
\begin{proof}
It is easy to verify the BiHom-associativity and BiHom-Lie color identities. It remains to prove the BiHom-Leibniz color identity.
For any $x, y, z\in\mathcal{H}(P)$,
 \begin{eqnarray}
  \{\alpha\beta(x), x\ast y\}
&=&[\theta\alpha\beta(x), \theta(y\ast z)]\nonumber\\
&=&[\theta\beta\alpha(x), \theta(\theta(y)\theta(z))]\nonumber\\
&=&[\alpha\beta\theta(x), \theta^2(y)\theta(z)]\nonumber\\
&=&[\beta\theta(x), \theta^2(y)]\beta\theta(z)+\varepsilon(x, y)\beta\theta^2(y)[\theta\alpha(x), \theta(z)]\nonumber\\
&=&\theta[\theta\beta(x), \theta(y)]\theta\beta(z)+\varepsilon(x, y)\beta\theta^2(y)[\theta\alpha(x), \theta(z)]\nonumber.
 \end{eqnarray}
The even linear map being an $\alpha^0$-averaging operator again, it comes
 \begin{eqnarray}
  \{\alpha\beta(x), x\ast y\}
&=&\theta[\theta\beta(x), \theta(y)]\theta\beta(z)+\varepsilon(x, y)\theta\beta\theta(y)\theta[\alpha(x), \theta(z)]\nonumber\\
&=&\theta[\theta\beta(x), \theta(y)]\theta\beta(z)+\varepsilon(x, y)\theta\Big(\theta\beta(y)[\theta\alpha(x), \theta(z)]\Big)\nonumber\\
&=&\theta[\theta\beta(x), \theta(y)]\theta\beta(z)+\varepsilon(x, y)\theta\beta(y)\theta[\theta\alpha(x), \theta(z)]\nonumber\\
&=&\{\beta(x), y\}\ast\beta(z)+\varepsilon(x, y)\beta(y)\ast\{\alpha(x), z\}\nonumber.
 \end{eqnarray}
This finishes the proof.
\end{proof}

Whenever, $\alpha=\beta$, we recover Theorem 2.11 \cite{BIA}.

\begin{theorem}
 Let $(P, \cdot, [-, -], \varepsilon, \alpha)$ be a  BiHom-Poisson color algebra and $\theta : P\rightarrow P$ be an injective  
$(\alpha^k, \beta^l)$-averaging operator. Then the new products
\begin{eqnarray}
 x\ast y=\theta(x)\cdot \alpha^k\beta^l(y)\quad\mbox{and}\quad \{x, y\}=[\theta(x), \alpha^k\beta^l(y)], \forall x, y\in \mathcal{H}(P)
\end{eqnarray}
makes $(P, \ast, \{-, -\}, \varepsilon, \alpha)$ a  BiHom-Poisson color algebra.\\
Moreover, $\theta$ is a morphism of  BiHom-Poisson color algebra of $(P, \ast,  \{-, -\}, \varepsilon, \alpha, \theta)$
 onto\\ $(P, \cdot,  [-, -], \varepsilon, \alpha, \theta)$.
 \end{theorem}
\begin{proof}
We leave the  checking of the BiHom-associativity and BiHom-Lie color identities to the reader.
Now, let us prove the Hom-Leibniz color identity; for any $x, y, z\in\mathcal{H}(P)$,
 \begin{eqnarray}
  \theta(\{\alpha\beta(x), y\ast z\})
&=&\theta[\theta\alpha\beta(x), \alpha^k\beta^l(y\ast z)]=\theta[\theta\beta\alpha(x), \alpha^k\beta^l(\theta(y)\alpha^k\beta^l(z))]\nonumber\\
&=&[\theta\beta\alpha(x), \theta(\theta(y)\alpha^k\beta^l(z))]=[\beta\alpha\theta(x), \theta(y)\theta(z)]\nonumber\\
&=&[\theta\beta(x), \theta(y)]\beta\theta(z) +\varepsilon(x, y)\beta\theta(y)\cdot[\theta\alpha(x), \theta(z)]\nonumber\\
&=&\theta[\theta\beta(x), \alpha^k\beta^l(y)]\beta\theta(z) +\varepsilon(x, y)\beta\theta(y)\cdot\theta[\theta\alpha(x), \alpha^k\beta^l(z)]\nonumber\\
&=&\theta\Big(\theta[\theta\beta(x), \alpha^k\beta^l(y)]\alpha^{k}\beta^{l+1}(z) +\varepsilon(x, y)\beta\theta(y)\cdot\beta^l\alpha^{k}[\theta\alpha(x), \beta^l\alpha^k(z)]\Big)\nonumber\\
&=&\theta\Big(\theta\{\beta(x), y\}\alpha^{k}\beta^{l+1}(z) +\varepsilon(x, y)\beta\theta(y)\cdot\alpha^{k}\beta^l\{\alpha(x), z\}\Big)\nonumber\\
&=&\theta\Big(\{\beta(x), y\}\ast\beta(z) +\varepsilon(x, y)\beta(y)\ast\{\alpha(x), z\}\Big)\nonumber.
 \end{eqnarray}
The second part comes from the definition of averaging operator.
The associativity and Hom-Jacobi identity are proved in the same way.
\end{proof}

Whenever, $\alpha=\beta$, we recover Theorem 2.13 \cite{BIA}.

\begin{definition}
Let $(P, \cdot, [-, -], \varepsilon, \alpha, \beta)$ be a  BiHom-Poisson color algebra. For any integers $k, l$,  an even linear map $\theta : P\rightarrow P$ is called
an element of $(\alpha^k, \beta^l)$-centroid on $P$ if
\begin{eqnarray}
\alpha\circ\theta&=&\theta\circ\alpha,\quad \beta\circ\theta=\theta\circ\beta,\\
 \theta(x)\cdot \alpha^k\beta^l(y)&=&\theta(x)\cdot\theta(y)=\alpha^k\beta^l(x)\cdot \theta(y)),\\ 
 {[\theta(x), \alpha^k\beta^l(y)]}&=&[\theta(x), \theta(y)]=[\alpha^k\beta^l(x), \theta(y)]),
\end{eqnarray}
for all $x, y\in\mathcal{H}(P)$.
 \end{definition}
The set of elements of centroid is called centroid.

\begin{theorem}
 Let $(P, \cdot, [-, -], \alpha, \beta)$ be a  BiHom-Poisson color algebra and $\theta : P\rightarrow P$ an element of 
$(\alpha^k, \beta^l)$-centroid. Then, for any $k, l\geq 0$
$$(P, \ast, \{-, -\}, \varepsilon, \alpha^{k+1}\beta^l, \alpha^k\beta^{l+1})$$
is a  BiHom-Poisson color algebra with the multiplications
$$x\ast y:=\alpha^k\beta^l(x)\cdot\theta(y)\quad\mbox{and}\quad \{x, y\}:=[\alpha^k\beta^l(x), \theta(y)]$$
for all $x, y\in\mathcal{H}(P)$.
\end{theorem}
\begin{proof}
Let us check the BiHom-associativity. For any homogeneous elements $x, y, z\in P$, we have
\begin{eqnarray}
 \alpha^{k+1}\beta^l(x)\ast(y\ast z)&=&\theta\alpha^{k+1}\beta^l(x)\cdot\theta(\theta(y)\cdot\theta(z))\nonumber\\
&=&\theta\alpha^{k+1}\beta^l(x)\cdot\alpha^{k}\beta^l(\theta(y)\cdot\theta(z))\nonumber\\
&=&\theta\alpha^{k+1}\beta^l(x)\cdot(\alpha^{k}\beta^l\theta(y)\cdot\alpha^{k}\beta^l\theta(z))\nonumber\\
&=&\alpha\theta(\alpha^{k}\beta^l(x))\cdot(\alpha^{k}\beta^l\theta(y)\cdot\alpha^{k}\beta^l\theta(z))\nonumber\\
&=&(\theta\alpha^{k}\beta^l(x)\cdot\alpha^{k}\beta^l\theta(y))\cdot\alpha^{k}\beta^{l+1}\theta(z))\nonumber\\
&=&\alpha^{k}\beta^l(\theta(x)\cdot\theta(y))\cdot\theta\alpha^{k}\beta^{l+1}(z))\nonumber\\
&=&(\theta(x)\cdot\theta(y))\ast\alpha^{k}\beta^{l+1}(z)\nonumber\\
&=&(x\ast y)\ast\alpha^{k}\beta^{l+1}(z)\nonumber.
\end{eqnarray}

Now let us prove the BiHom-Jacobi identity,
\begin{eqnarray}
 \{(\alpha^k\beta^{l+1})^2(x), \{(\alpha^k\beta^{l+1})(y), (\alpha^{k+1}\beta^l)(z)\}\}
&=&[\theta(\alpha^k\beta^{l+1})^2(x), \theta[\theta\alpha^k\beta^{l+1}(y), \theta\alpha^{k+1}\beta^{l}(z)]] \nonumber\\
&=&[\theta(\alpha^k\beta^{l+1})^2(x), \alpha^k\beta^{l}[\alpha^k\beta^{l}\theta\beta(y), \alpha^{k}\beta^{l}\theta\alpha(z)]] \nonumber\\
&=&\alpha^{2k}\beta^{2l}[\beta^2\theta(x), [\beta\theta(y), \alpha\theta(z)]] \nonumber.
\end{eqnarray}
Thus,
\begin{eqnarray}
 \oint \varepsilon(z, x)\{(\alpha^k\beta^{l+1})^2(x), \{(\alpha^k\beta^{l+1})(y), (\alpha^{k+1}\beta^l)(z)\}\}\\
\qquad=\alpha^{2k}\beta^{2l}\oint\varepsilon(z, x)[\beta^2\theta(x), [\beta\theta(y), \alpha\theta(z)]]=0.\nonumber
\end{eqnarray}
Finally let us check the BiHom-Leibniz color identity :
\begin{eqnarray}
 \{\alpha^{k+1}\beta^l\alpha^k\beta^{l+1}(x), y\ast z\}
&=&\{\alpha^{2k}\beta^{2l}\alpha\beta(x), y\ast z\}\nonumber\\
&=&[\theta(\alpha^{2k}\beta^{2l}\alpha\beta(x)), \theta(\theta(y)\theta(z))]\nonumber\\
&=&[\alpha^{2k}\beta^{2l}\alpha\beta\theta(x), \alpha^k\beta^l(\theta(y)\theta(z))]\nonumber\\
&=&\alpha^{k}\beta^{l}[\alpha\beta(\theta\alpha^{k}\beta^{l}(x)), \theta(y)\theta(z)]\nonumber\\
&=&\alpha^{k}\beta^{l}\Big([\beta(\theta\alpha^{k}\beta^{l}(x)), \theta(y)]\beta\theta(z)
+\varepsilon(x, y)\beta\theta(y)[\alpha(\theta\alpha^{k}\beta^{l}(x)), \theta(z)]\Big)\nonumber\\
&=&\alpha^{k}\beta^{l}\Big([\theta(\alpha^{k}\beta^{l+1}(x)), \theta(y)]\beta\theta(z)
+\varepsilon(x, y)\beta\theta(y)[\alpha(\theta\alpha^{k}\beta^{l}(x)), \theta(z)]\Big)\nonumber\\
&=&\alpha^{k}\beta^{l}[\theta(\alpha^{k}\beta^{l+1}(x)), \theta(y)]\cdot\theta\alpha^{k}\beta^{l+1}(z)\nonumber\\
&&+\varepsilon(x, y)\theta\alpha^{k}\beta^{l+1}(y)\cdot\alpha^{k}\beta^{l}[\theta\alpha^{k+1}\beta^{l}(x), \theta(z)]\nonumber\\
&=&\theta\{\alpha^{k}\beta^{l+1}(x), y\}\cdot\theta\alpha^{k}\beta^{l+1}(z)\nonumber\\
&&+\varepsilon(x, y)\theta\alpha^{k}\beta^{l+1}(y)\cdot\theta\{\alpha^{k+1}\beta^{l}(x), z\}\nonumber\\
&=&\{\alpha^{k}\beta^{l+1}(x), y\}\ast\alpha^{k}\beta^{l+1}(z)
+\varepsilon(x, y)\alpha^{k}\beta^{l+1}(y)\ast\{\alpha^{k+1}\beta^{l}(x), z\}\nonumber.
\end{eqnarray}
This conclude the proof.
\end{proof}
\subsection{BiHom-Poisson color algebras induced by multipliers}
Let $P$ be a  BiHom-Poisson color algebra and 
 $\sigma : G\times G\rightarrow \mathbb{K}^*$ be a {\it symmetric multiplier} i.e.
\begin{enumerate}
 \item [i)] $\sigma(g, g')=\sigma(g', g), \forall x, y\in G$
\item [ii)] $\sigma(g, g')\sigma(g'', g+g')$ is invariant under cyclic permutation of $g, g', g''\in G$.
\end{enumerate}

\begin{theorem}
Let $(P, \ast, [-, -], \varepsilon, \alpha)$ be a  BiHom-Poisson color algebra and 
 $\delta : G\times G\rightarrow \mathbb{K}^*$ be the bicharacter associated with the   {multiplier} $\sigma$ on $G$ 
Then, $(P, \ast^\sigma, [-, -]^\sigma, \varepsilon\delta, \alpha)$ is also a  BiHom-Poisson color algebra with
$$x\ast^\sigma y=\sigma(x, y)x\ast y,\quad[x, y]^\sigma=\sigma(x, y)[x, y]\quad\mbox{and}\quad \varepsilon\delta(x, y)
=\varepsilon(x, y)\sigma(x, y)\sigma(y, x)^{-1},$$
for any $x, y\in \mathcal{H}(P)$.\\
Moreover, an endomorphism of $(P, \cdot,  [-, -], \varepsilon, \alpha, \beta)$ is also an endomorphism of 
$(P, \ast^\sigma,  [x, y]^\sigma, \varepsilon, \alpha, \beta)$.
\end{theorem}
\begin{proof}
For any homogeneous elements $x, y, z\in P$,
 \begin{eqnarray}
 [\alpha\beta(x), y\ast^\sigma z]^\sigma
&=&[\alpha\beta(x),  \sigma(y, z)y\ast z]^\sigma\nonumber\\
&=&\sigma(y, z)\sigma(x, y+z)[\alpha\beta(x), y\ast z]\nonumber\\
&=&\sigma(y, z)\sigma(x, y+z)\Big([\beta(x), y]\ast\beta(z)+\varepsilon(x, y)\beta(y)\ast[\alpha(x), z]\Big)\nonumber\\
&=&\sigma(y, z)\sigma(x, y+z)[\beta(x), y]\ast\beta(z)\nonumber\\
&&\quad+\varepsilon(x, y)\sigma(y, z)\sigma(x, y+z)\beta(y)\ast[\alpha(x), z]\nonumber.
 \end{eqnarray}
Using twice the fact that $\sigma$ is a multiplier, we have
\begin{eqnarray}
 [\alpha\beta(x), y\ast^\sigma z]^\sigma
&=&\sigma(y, z)\sigma(x, y+z)[\beta(x), y]\ast\beta(z)+\sigma(x, y)\sigma(z, y+x)\varepsilon(x, y)\beta(y)\ast[\alpha(x), z]\nonumber\\
&=&\sigma(y, z)\sigma(x, y+z)[\beta(x), y]\ast\beta(z)\nonumber\\
&&\quad+\varepsilon(x, y)\sigma(x, y)\sigma(y, x)^{-1}\sigma(z, x+y)\sigma(y, x)\beta(y)\ast[\alpha(x), z]\nonumber\\
&=&\sigma(y, z)\sigma(x, y+z)[\beta(x), y]\ast\beta(z)\nonumber\\
&&\quad+\varepsilon\delta(x, y)\sigma(y, x+z)\sigma(x, z)\beta(y)\ast[\alpha(x), z]\nonumber\\
&=&[\beta(x), y]^\sigma\ast^\sigma\beta(z)+\varepsilon\delta(x, y)\beta(y)\ast^\sigma[\alpha(x), z]^\sigma.\nonumber
 \end{eqnarray}
This ends the proof.
\end{proof}

\begin{corollary}
If $(P, \ast, [-, -], \varepsilon, \alpha, \beta)$ is a  BiHom-Poisson color algebra and $\sigma$ a {\it symmetric multiplier} on $G$. Then
 $(P, \ast^\sigma, [-, -]^\sigma, \varepsilon, \alpha, \beta)$ is a  BiHom-Poisson color algebra with
$$x\ast^\sigma y=\sigma(x, y)x\ast y\quad\mbox{and}\quad[x, y]^\sigma=\sigma(x, y)[x, y],$$
for any $x, y\in \mathcal{H}(P)$.
\end{corollary}

Similar constructions may be made for BiHom-Poisson color algebra of arbitrary degree which will include BiHom-Gerstenhaber color algebras.
%%%%%%%%%%%
% {\bf Acknowlegments :} 

\end{document}